\documentclass[leqno]{amsart} 

\usepackage{pdfsync}
\usepackage{dsfont}
\usepackage{color}
\usepackage{amsthm} 
\usepackage{braket}
\usepackage{bm}
\usepackage{amsmath}
\usepackage{graphicx}
\usepackage{subfigure}
\usepackage{caption}
\usepackage{comment}
\usepackage{pb-diagram}
\usepackage{amsfonts}
\usepackage{amssymb}%
\usepackage{times}
\usepackage{microtype}

\providecommand{\U}[1]{\protect\rule{.1in}{.1in}}

\newtheorem{counter}{Counter}[section]
\newtheorem{theorem}[counter]{Theorem}

\newtheorem{lemma}[counter]{Lemma}
\newtheorem{proposition}[counter]{Proposition}

\theoremstyle{remark}
\newtheorem{remark}[counter]{Remark}

\newcommand{\x}{\mathbf{x}}
\newcommand{\1}{\mathbf{1}}
\newcommand{\Prob}{\mathbb{P}}
\newcommand{\Rd}{\mathrm{d}}

\newcommand{\tubo}{T_{\veps}^+}
\newcommand{\tubno}{T_{\veps_n}^+}
\newcommand{\tubi}{T_{\veps}^-}
\newcommand{\tubni}{T_{\veps_n}^-}

\newcommand{\R}{\mathbb{R}}
\newcommand{\NN}{\mathbb{N}}
\newcommand{\E}{\mathbb{E}}

\newcommand{\cc}{C}

\newcommand{\D}{D}
\newcommand{\dom}{D}

\newcommand{\te}{\textrm}

\newcommand{\veps}{\varepsilon}

\newcommand{\GP}[1]{\mathrm{GPer}_{n,\veps_n}(#1)}
\newcommand{\GPe}[1]{\GPem_{n,\veps_n}(#1)}

\newcommand{\NLTVe}[1]{ \Pem_{\veps_n}( #1) }
\newcommand{\NLTVenon}[1]{ \Pem_{\veps}( #1) }
\newcommand{\NLTVen}[1]{ \Pem_{\veps_n}( #1) }
\newcommand{\BV}[1]{ \Pem(#1) }

\newcommand{\red}{\color{red}}

\newcommand{\nc}{\normalcolor}

\DeclareMathOperator{\Var}{Var}

\DeclareMathOperator{\Pem}{Per}
\DeclareMathOperator{\GPem}{GPer}
\DeclareMathOperator{\dist}{dist}
\DeclareMathOperator{\divergence}{div}
\DeclareMathOperator{\error}{std}
\DeclareMathOperator{\cut}{Cut}

\begin{document}

\title{Estimating perimeter using graph cuts}
 
\author{ Nicol\'as Garc\'ia Trillos$^1$, Dejan Slep\v{c}ev$^2$ and James von Brecht$^3$}

\address{$^1$ Division of Applied Mathematics, Providence, RI, 02912, USA.\\ 
email:  nicolas\_garcia\_trillos@brown.edu}

\address{$^2$
Department of Mathematical Sciences, Carnegie Mellon University, Pittsburgh, PA, 15213, USA. \\
tel. +412 268-2545, 
email: slepcev@math.cmu.edu }

\address{$^3$ Department of Mathematics and Statistics, California State University, Long Beach
Long Beach, CA 90840, USA. \\
email: James.vonBrecht@csulb.edu
}

\newcounter{broj1}
\date{\today}
\maketitle

\begin{abstract}
We investigate the estimation of  the perimeter of a set by a graph cut of 
a random geometric graph. For $\Omega \subset D = (0,1)^d$, with $d \geq 2$, we are given
 $n$ random i.i.d. points on $D$ whose membership in $\Omega$ is known. 
 We consider the sample as a random geometric graph with connection distance $\veps>0$.
We estimate the perimeter of $\Omega$ (relative to $D$) by the, appropriately rescaled,  graph cut between the vertices in  $\Omega$ and the vertices in $D \backslash \Omega$.
We obtain bias and variance estimates on the error, which are optimal in scaling with respect to $n$ and $\veps$. We consider two scaling regimes: the dense (when the average degree of the vertices goes to $\infty$) and the sparse one (when the degree goes to $0$). 
In the dense regime there is a crossover in the nature of approximation at dimension $d=5$: we show that in low dimensions $d=2,3,4$ one can obtain confidence intervals for the approximation error, while in higher dimensions one can only obtain error estimates for testing the hypothesis that the perimeter is less than a given number. 
\end{abstract}


\section{Introduction}

This paper investigates the use of random-graph cuts to obtain empirical estimates of the perimeter of a domain $\Omega \subset D := (0,1)^d$ for $d\geq 2$. 
Let $\x_1, \dots, \x_n,\ldots$ denote a sequence of independent random points uniformly distributed on the unit cube $D$  and let $V_n:=\left\{\x_1, \dots, \x_n \right\}$. 

The problem of estimating the perimeter of $\Omega$ based on knowing which points of $V_n$ belong to $\Omega$ is a classical question, see \cite{ ArCuFr09, CuFrGy13, CuFrRo07, JinYuk11, Neeman, PatRod08} for recent contributions, and see Subsection \ref{Discussion} below which contains a discussion about related work.  Here we consider an estimator of the perimeter that is based on a geometric graph constructed from the point cloud $V_n$. More precisely, we select $\veps_n>0$, and connect two points in the cloud if they are within distance $\veps_n$ of each other; then we consider an appropriately scaled `cut' determined by the number of edges in the graph that connect points in $\Omega$ with points that belong to $\Omega^c$. This type of estimator is natural to consider, since graph cuts arise as a discretization of the perimeter in many applications such as clustering \cite{ACPP,  BLUV12, HagKah, HeinSetz, KanVemVet04, ShiMalik,  SzlamBresson,  vonLux_tutorial,  WeiChe89}. Our choice of the estimator is therefore based on its use in various statistical and machine learning applications.

We focus on estimating the approximation error of the perimeter of an arbitrary (but fixed) set $\Omega$; the error estimates that we obtain are uniform on a class of sets where certain geometric quantities are controlled (see Remark \ref{remarkUniform} below). One of the important features of our estimator is that it has a small bias. Indeed, the expectation of our estimator provides a second order approximation (in terms of the natural parameter $\veps_n$) of the true perimeter of $\Omega$ under some regularity conditions on the boundary of the set (see \eqref{bias1} below); this turns out to be a sharp estimate for the bias. We also obtain precise estimates for the variance of the estimator; for these estimates to hold, we do not need any regularity assumptions on the boundary of the set $\Omega$ except that it has finite perimeter in the most general sense. Furthermore, we show that our estimator converges a.s. for remarkably sparse graphs (and indeed in settings which are 
sparser than for previously considered estimators). The a.s. convergence holds 
with no regularity assumptions on $\Omega$ (other than the fact that it has finite perimeter in the 
most general sense). Finally, assuming some smoothness on the boundary of $\Omega$, we are able to establish (in the dense graph regime $\frac{1}{n^{1/d}} \ll \veps \ll1$) the asymptotic distribution of the error.
These estimates lead to asymptotic confidence intervals (which we refer to simply as \emph{confidence intervals}) and bounds on the type I and type II errors for hypothesis tests associated to the perimeter of a set. 
\nc
\bigskip

\subsection{Set-up and main results}
Let us now be more precise about the setting we consider in this paper. We consider random geometric graphs with vertex set $V_n$ and radius $\veps_n >0$. That is, graphs where
 $\x_i$ and $\x_j$ are connected by an edge if  $||\x_i- \x_j||\leq \veps_n$. The \emph{graph cut} between $A \subseteq V_n$ and $A^c$ is given by 
$$ \cut_{\veps_n}(A,A^{c}) := \sum_{\x_i \in A}\sum_ {\x_j \in V \setminus A} \; \1_{ \{ \|\x_i - \x_j\| \leq \veps_n \} }.$$
We define the \emph{graph perimeter} as a rescaling of the graph cut: For any $\Omega \subseteq D$
\begin{equation}
\GPe{\Omega} := \frac{2 }{n (n-1) \veps^{d+1}_n} \cut(V_n \cap \Omega, V_n \cap \Omega^c).
\label{RescaledGraphPerimeter}
\end{equation}
See Figure 1 for an illustration of this construction.
The scaling is chosen so that $\GPe{\Omega}$ becomes a consistent estimator for the true (continuum) perimeter.
\begin{figure}[!h]
\begin{center}
{\includegraphics[width=7cm]{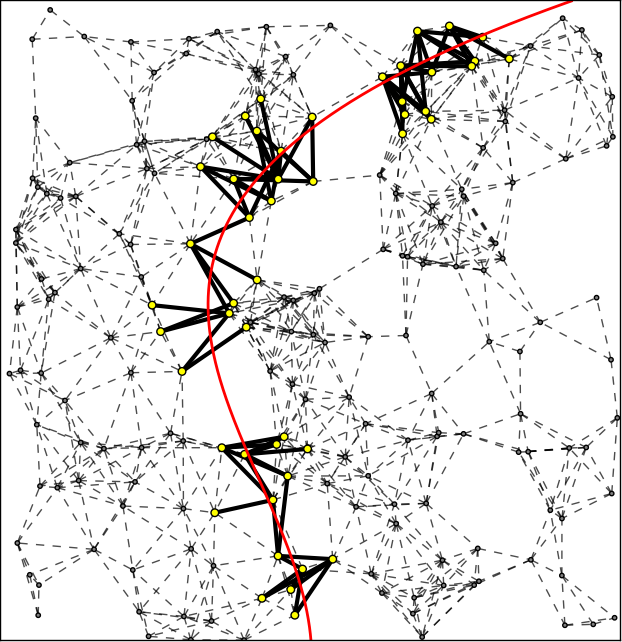}
\put(-180,85){\red \LARGE $\Omega$}
\put(-66,85){\red \LARGE $\Omega^c$}}
\caption{Edges  of the cut between $\Omega$ and $\Omega^c$ are represented by bold lines, while other edges are dashed lines. Total number of vertices is $n=200$ and connectivity radius $\veps=0.13$.
}
\end{center}
\label{fig:cut1}
\end{figure}




%
%

One of our interests  lies in determining how well does $\GP{\Omega}$ estimate the relative perimeter 
of $\Omega$ in $D$. We first investigate for which scaling of $\veps_n$ on $n$ does the convergence hold
almost surely as $n \to \infty$. 
In other words, we want to understand the relation between edge-sparsity of random geometric graphs and point-wise convergence of the graph perimeter to the continuum perimeter in the almost sure sense. 
We consider this question for a very broad family of sets $\Omega \subset \dom$, which are only assumed to have finite relative perimeter in $D$ in the general sense of \cite{AFP}. 
That is, we define the relative perimeter of $\Omega$ with respect to $\dom$ to be
\begin{equation} \label{perimeter}
\BV{\Omega} =  \sup \left\{  \int_{\Omega} \divergence(v) \; \Rd x \: : \; (\forall x \in D ) \;\,  \| v(x)\| \leq 1, \: \:  \:  v \in C^\infty_c(\dom, \R^d)   \right\}.
\end{equation}
If $\Omega$ has a smooth relative boundary then $\BV{\Omega}$ is nothing but the surface area of 
$\partial \Omega \cap \dom$. We remark that the notion of the perimeter we use is more general than the notion of Hausdorff measure of the boundary, $\mathcal H^{d-1}(\partial \Omega \cap D)$, and than the Minkowski content, which are the ones more typically used in the statistics literature \cite{ArCuFr09, CuFrGy13, CuFrRo07}. In
particular, as we see below, we work with consistent nonparametric estimators in the most general setting available.


As we recall below in (\ref{eq:tri_mean} --\ref{MeanConvergingToPerimeter}), it is known that when  $\veps_n \to 0$ as $n \to \infty$ then the bias of the estimator vanishes in the limit: 
\begin{equation} \label{temp1}
\E(\GP{\Omega})  \to \sigma_{d} \BV{\Omega} \qquad \te{as } \qquad n \to \infty.
 \end{equation}
The scaling factor $\sigma_{d}$ satisfies 
\begin{equation} \label{sigma}
 \sigma_{d}:= \int_{||z|| \leq 1} |z_1| \;\Rd z = \frac{2s_{d-2} }{(d+1)(d-1)},
\end{equation}
where $z_1$ denotes the first component of the vector $z \in \mathbb{R}^{d}$ and 
$s_{d-2}$ is the area of the $(d-2)$-dimensional unit sphere (the boundary of the unit ball in $\R^{d-1}$).
We refer to the normalizing quantity $\sigma_{d}$ as the \emph{surface tension}.

We obtain the following estimates on the deviation of the graph perimeter $\GP{\Omega}$ from its mean.
Let
\begin{equation}\label{eq:rate_def}
f(n, \veps_n) :=
\begin{cases}
\frac{1}{\sqrt{n \veps_n}}  \;  &  \te{if } \;\frac{1}{n^{1/d}} \leq \veps_n  \medskip \\
\frac{1}{n \veps_n^{(d+1)/2}}  & \te{if } \;\frac{1}{n^{2/(d+1)}} \leq \veps_n \leq \frac{1}{n^{1/d}}.
\end{cases}
\end{equation}

%
\begin{theorem}
Let $p \geq 1$ and let $\Omega \subseteq \D$ be a set with finite perimeter. Assume $\veps_n \to 0$ as $n \to \infty$. Then,
\begin{equation}
 \E(|\GPe{\Omega} - \E(\GP{\Omega})|^p) \leq \cc_{p,d} \, (  \max\{1,\BV{\Omega}\} f(n,\veps_n))^p
  \label{MomentsTheorem}
\end{equation}
%
%
where $\cc_{p,d} $ is a constant that depends only on $p$ and dimension $d$. In particular, if $n ^{-\frac{2}{(d+1)}} \ll \veps_n \ll 1$, then
$$ \GP{\Omega} \rightarrow   \sigma_{d} \BV{\Omega},  \: \te{almost surely as }  n \to \infty.   $$
\label{TheoremGeneral}
\end{theorem}
The last part of the previous theorem follows from \eqref{temp1}, the moment estimates \eqref{MomentsTheorem}, Markov's inequality, and Borel-Cantelli Lemma which imply that
$$ \GPe{\Omega}- \E\left( \GPe{\Omega} \right) \rightarrow 0 \quad \text{a.s.}   $$

\begin{remark}
We note that the a.s. convergence holds for rather sparse graphs (see Figure \ref{fig:cutsp}). Namely the typical degree of a node is 
$\omega_d n \veps^d $, where $\omega_d$ is the volume of the unit ball in $d$ dimensions. When 
$n ^{-\frac{2}{(d+1)}} \ll \veps_n \ll n^{-\frac{1}{d}}$ the a.s. convergence holds, while the average degree of a vertex converges to zero. The convergence is still possible because the expected number of edges crossing $\partial \Omega$ is still a quantity converging to infinity. 
\end{remark}

\begin{figure}
\subfigure[$n=400$ and $\veps=0.045$\nc]{\includegraphics[width=60mm]{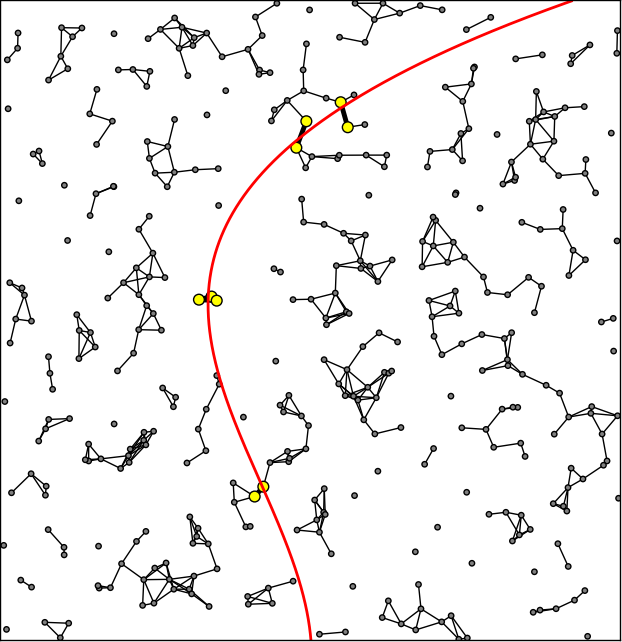}
}
\label{figureE}
\hspace*{2mm}
\subfigure[$n=1000$ and $\veps=0.027$ ]{ \includegraphics[width=60mm]{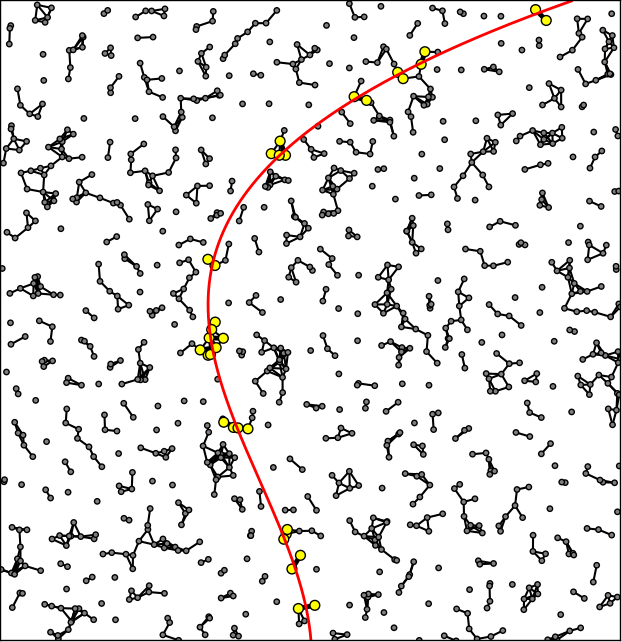}
}
\label{figureF}
\caption{Here we illustrate the ``sparse" regime when $1/n^{1/d}  \gg \veps  \gg (1/n)^{2/d+1} $ when the average degree goes to zero. Nevertheless since the number of edges in the cut still increases as $n$ does the convergence of the cut to the perimeter still holds. \nc}  
\label{fig:cutsp}
\end{figure}

\begin{remark}
Given that we show the almost-sure consistency of our estimators for any arbitrary (but fixed) set for which the perimeter is finite, our construction provides a universal strongly convergent estimator which was the desired property listed as an open problem in \cite{CuFrGy13} (for the estimator they considered).
\end{remark}

We  turn to estimating the bias of the empirical approximation:
$|\E(\GP{\Omega}) - \sigma_{d} \Pem(\Omega)| $. We first characterize the mean of the graph perimeter, $\E(\GP{\Omega})$, as the \emph{non-local perimeter} $\Pem_{\veps_n}(\Omega)$ of $\Omega,$ defined as
\begin{equation}
 \NLTVenon{\Omega}: = \frac{2}{\veps^{d+1}}\int_{\Omega} \int_{\dom \backslash \Omega} \1_{\{\|  x -y\| \leq \veps \}} \; \Rd x \Rd y . 
\label{NonLocalPerimeter}
\end{equation}
The non-local nature of the functional essentially has to do with the fact that it involves averages of finite differences as opposed to a \emph{local} approach where one considers derivatives.  We then proceed to estimate $\left|\Pem_{\veps_n}(\Omega)- \sigma_{d} \Pem(\Omega)\right|$ explicitly. It proves straightforward to check that $\left|  \Pem_{\veps_n}(\Omega)- \sigma_{d} \Pem(\Omega)\right| = O\left(\veps_n  \right)$ for general subsets $\Omega \subseteq \dom$ with smooth relative boundary. However, we show that the error is actually quadratic in $\veps_n$
\begin{equation} \label{bias1}
|\E(\GP{\Omega}) - \sigma_{d} \Pem(\Omega)| = \left|\Pem_{\veps_n}(\Omega)- \sigma_{d} \Pem(\Omega)\right| = O(\veps_n^2) 
\end{equation}
under the extra condition that $\dist ( \Omega, \partial D) >0$. This is the content of the next lemma, whose proof may be found in Appendix \ref{AppendixLemmaBias}.
\begin{lemma} \label{pnlp}
Let $\Omega$ be a set with smooth boundary, such that $\dist(\Omega ,\partial \D) >0$. Let $0 < \veps < \dist(\Omega ,\partial \D)$ and let $\NLTVenon{\Omega}$ be defined by \eqref{NonLocalPerimeter}. Then
\begin{equation}
\NLTVenon{\Omega}= \sigma_{d} \Pem(\Omega) + O(\veps^2).
\label{errorNonlocal}
\end{equation}
\end{lemma}

\begin{remark}
The assumption $\dist(\Omega, \partial D) >0$ in the above Lemma is needed in order to obtain bias of order $\veps^2$. If $\Omega$ touches the boundary $\partial D$ the error of order $\veps^2$ is not expected, as can be seen for example by considering the rectangle $\Omega := \left\{ x=(x_1,\dots,x_d) \in D \: : \: x_1 \leq 1/2 \right\}$, for which the error is of order $\veps$; in this situation the error is completely due to the region where $\partial \Omega$ meets $\partial D$ transversally.
Thus, in general, for $\Omega \subseteq D$ with smooth relative boundary, the bias is of order $\veps$. On the other hand, the smoothness of the boundary of $\Omega$ is only needed in the previous lemma to guarantee that curvature and its derivatives are well defined. Finally, the constant involved in the term $O(\veps^2)$ depends on the reach of the set $\Omega,$ and the intrinsic curvature of $\partial \Omega$ together with its derivatives; this can be seen from our computations in Appendix \ref{AppendixLemmaBias}.
\end{remark}

Combining the bias and variance of estimates allows us 
to obtain the rates of convergence for the error $|\GP{\Omega} - \sigma_{d} \Pem(\Omega)|$.
In particular we estimate  the  `standard deviation'
\begin{equation*}
 \error (n):= \E\left( (\GP{\Omega} - \sigma_{d} \Pem(\Omega))^2    \right)^{1/2},
\end{equation*}
which we may quantify precisely by using the variance-bias decomposition
\begin{equation*}
\begin{split}
\error^2 (n) &= \Var(\GP{\Omega} )  +  \left(   \E(\GP{\Omega}  ) - \sigma_{d} \Pem(\Omega) \right)^2.
\end{split}
\end{equation*}
Using the  special case $p=2$ of Theorem \ref{TheoremGeneral} to estimate the variance and using Lemma \ref{pnlp} to estimate the bias we obtain the following.
\begin{theorem}
\label{th:toter}
Let $\Omega \subset \D$ be an open set with smooth boundary.  Assume $n^{-\frac{2}{d+1}} \ll \veps_n \ll 1$ and
consider  $f(n,\veps_n)$  defined via \eqref{eq:rate_def}. The error of approximating $\sigma_{d} \Pem(\Omega)$ by $\GP{\Omega}$ satisfies
\begin{equation*}
\error (n)  =  O ( f(n,\veps_n) + \veps_n).
\end{equation*}
If we furthermore assume that $\Omega$ does not touch the boundary of $D$, that is  $\dist(\Omega, \partial \dom) >0$, then a better estimate holds:
\begin{equation*}
\error (n)  =  O ( f(n,\veps_n) + \veps_n^2 ).
\end{equation*}
\end{theorem}

A simple calculation using  \eqref{eq:rate_def} allows one to choose a scaling of $\veps_n$ on $n$ so that the error of the approximation is as small as possible. In particular, for a set $\Omega$ which touches the boundary ($\dist(\Omega, \partial D)=0$) the optimal scaling of $\veps_n$ is
\[ \veps_n \sim \begin{cases} 
n^{-1/3}  \quad &\te{if } d \leq 3 \\
n^{-2/(d+3)}  & \te{if } d \geq 3
\end{cases}  
\quad \te{ giving } \quad
\error(n) =  \begin{cases} 
n^{-1/3}  \quad &\te{if } d \leq 3 \\
n^{-2/(d+3)}  & \te{if } d \geq 3.
\end{cases}  
\]
We note that for $d < 3$ the optimal $\veps_n$ is achieved in the  regime $n^{-1/d} \lesssim \veps_n \ll 1$, while if  $d>3$ it is in the sparse regime. If we consider sets $\Omega$ with smooth boundary but such that $\dist(\Omega, \partial D)>0$, then the optimal scaling of $\veps_n$ on $n$ is as follows
\[ \veps_n \sim \begin{cases} 
n^{-2/5}  \quad &\te{if } d \leq 5 \\
n^{-4/(d+5)}  & \te{if } d \geq 5
\end{cases}  
\quad \te{ giving } \quad
\error(n) =  \begin{cases} 
n^{-1/5}  \quad &\te{if } d \leq 5 \\
n^{-2/(d+5)}  & \te{if } d \geq 5.
\end{cases}  
\]
Again we note that the optimal $\veps_n$ is in the sparse regime if $d >5$. This has implications to how well is the perimeter estimated by graph cuts in the graphs 
considered in most machine learning applications. Namely if $d \geq 5$ and the graph has average degree bounded from below, that is when  $n^{-1/d} \lesssim \veps_n \ll 1$, then, since the bias bound is sharp, most of the error is due to the bias term.

\begin{remark}
The constant appearing in the moment estimates of Theorem \ref{TheoremGeneral} depends exclusively on the power $p$, the dimension $d$, and the true perimeter of the set $\Omega$; in particular, taking $p=2 $ we see that the estimates for the variance of $\GP{\Omega}$ are uniform on the class of sets $\Omega$ whose perimeter is bounded above by some fixed constant. For general sets with  large perimeter we obtain uniform estimates for moments of relative error $|\GPe{\Omega} - \E(\GP{\Omega})|/\BV{\Omega}$ instead of the absolute error. Thus, by combining Theorem \ref{TheoremGeneral} and the proof of Lemma  \ref{pnlp} one can derive that 
the corresponding error estimates for $\lvert \GP{\Omega} - \sigma_d \Pem(\Omega)\rvert$ (either relative or absolute) are uniform on the class of sets $\Omega$ with smooth boundary satisfying the following conditions: The reach of $\Omega$ is bounded below by a fixed positive constant; The distance to the boundary $\partial D$ is bounded below by a fixed positive constant; The curvature and first derivatives of curvature are bounded from above by a fixed constant. This last requirement comes from the $O(\veps^2)$ terms following Taylor's theorem in the proof of our bias estimates.
\label{remarkUniform}
\end{remark}
\nc

%
 \medskip

We now consider obtaining confidence intervals for the value of the true perimeter $\Pem(\Omega)$ based on the estimator $\frac{1}{\sigma_d}\GPe{\Omega}$. We focus on the dense regime (See Figures \ref{fig:cutd} and \ref{fig:cutcd}), $\frac{1}{n^{1/d}}\ll \veps_n \ll 1$, and first obtain the asymptotic distribution of $\GPe{\Omega}- \E(\GPe{\Omega})$.  

\begin{figure}[!h]
\subfigure[$n = 100$, $\veps = 0.17$]{\includegraphics[width=60mm]{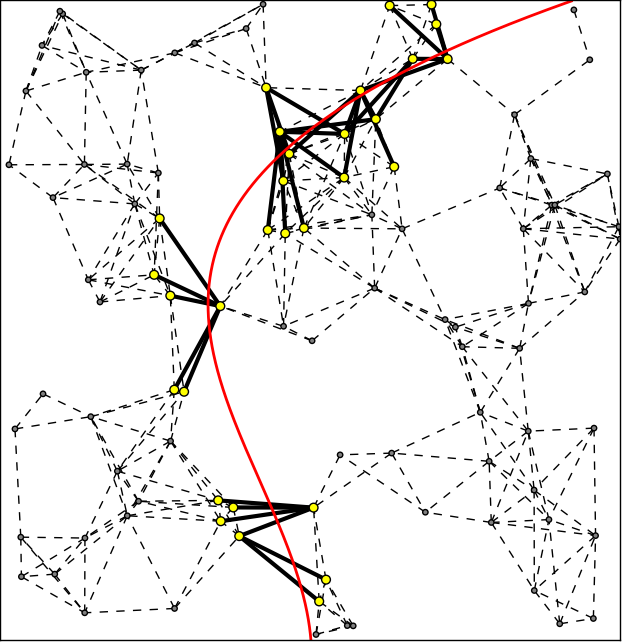}
}
\label{figureA}
\hspace*{2mm}
\subfigure[$n=300$ and $\veps=0.12$ 
]{\includegraphics[width=60mm]{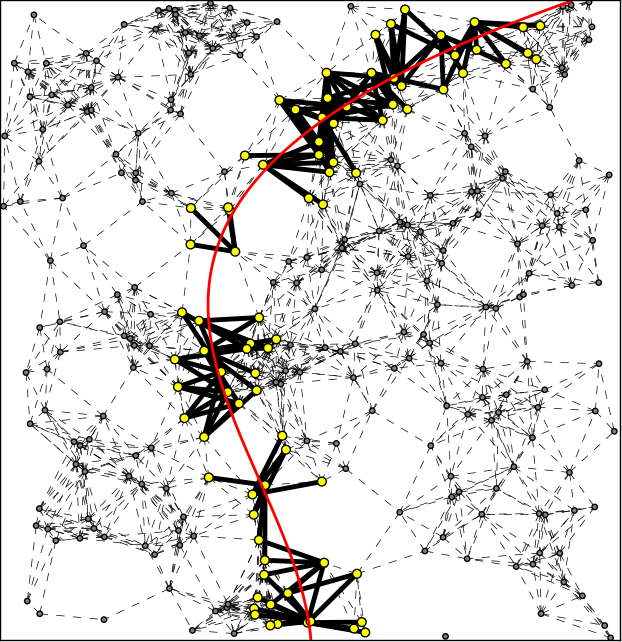}
}
\label{figureB}
\caption{Here we illustrate the ``very dense" regime, $1 \gg \veps  \gg (\ln(n)/n)^{1/d}$, when the graphs are connected with high probability. 
 } 
\label{fig:cutd}
\end{figure}

Since the term $\E \left( \GPe{\Omega} \right)= \Pem_{\veps_n}(\Omega)$ depends on the set $\Omega $ itself, which is unknown (we only assume we have an oracle access to it), we need the bias to be negligible compared to the standard deviation of our estimator. This allows us to construct confidence intervals for $\Pem(\Omega)$ without using any additional information about the bias (e.g. upper bounds).  From \eqref{MomentsTheorem} the  standard deviation of $\GPe{\Omega} $  scales as $\frac{1}{\sqrt{n \veps_n}}$ while by \eqref{bias1} the bias scales as $\veps_n^2$; these estimates lead to restrictions on the dimensions for which the bias is negligible with respect to the standard deviation. Namely this is possible for $d=2, 3,$ and $4$.

\begin{theorem}\label{AsympDistribFirstRegime} Let $\Omega\subseteq \D$ be an open set with smooth boundary such that $\dist(\Omega, \partial \D )>0$. Let $\veps_n$ be such that 
$$  \frac{1}{n^{1/d}} \ll \veps_n \ll 1 $$
Then,
$$ \sqrt{\frac{n \veps_n }{4 C_d \Pem(\Omega)}} \left( \GPe{\Omega} - \Pem_{\veps_n}(\Omega)  \right) \overset{w}{\longrightarrow}N(0,1), $$
where $C_d$ is given by 
\begin{equation}
C_d:= 2\int_{0}^{1} \left | B_d(0,1) \cap \left\{ x=(x_1,\dots, x_d) \in \R^d \: : \: x_d \geq t \right\} \right |^2 \Rd t.
\label{Cd}
\end{equation} 
If in addition the dimension $d $ is either $2,3$ or $4$ and if
\begin{equation} 
\frac{1}{n^{1/d}} \ll \veps_n \ll \frac{1}{n^{1/5}}
\label{conditionBiased}
\end{equation}
then,
\begin{equation}
\sqrt{\frac{n \veps_n }{ 4 C_d \Pem(\Omega)}} \left( \GPe{\Omega} -\sigma_d \Pem(\Omega)  \right) \overset{w}{\longrightarrow}N(0,1).
\label{eqn2}
\end{equation}
\end{theorem}

\begin{figure}
\subfigure[$n=100$ and $\veps=0.12$]{ \includegraphics[width=59mm]{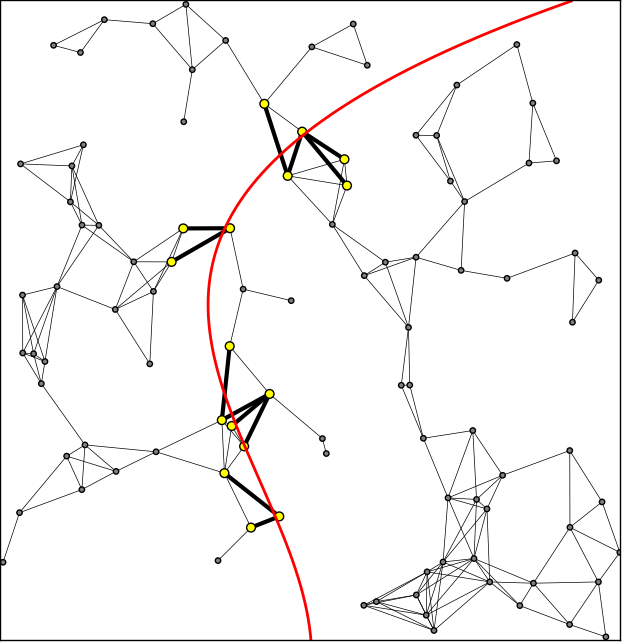}
}
\label{figureC}
\hspace*{2mm}
\subfigure[$n=400$ and $\veps=0.07$]{ \includegraphics[width=59mm]{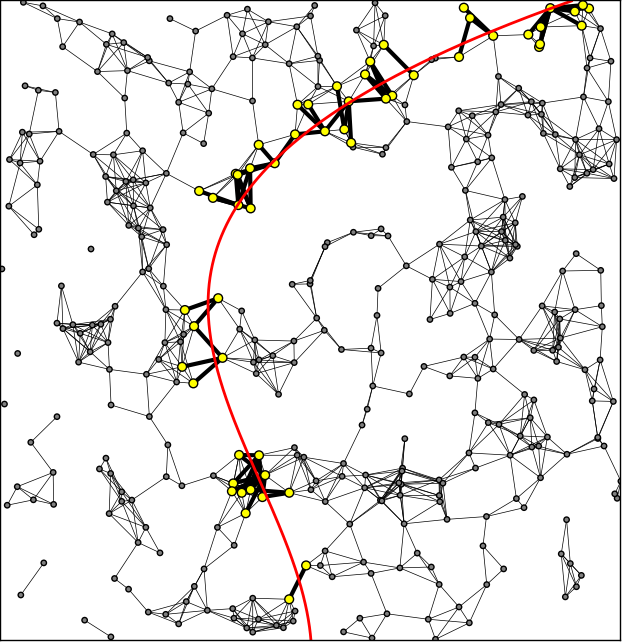}
}
\label{figureD}
\caption{Here we illustrate an intermediate, "dense", regime $(\ln(n)/n)^{1/d} \gg \veps   \gg (1/n)^{1/d}$. The average degree still increases, but the graphs are disconnected with high probability.\nc} 
\label{fig:cutcd}
\end{figure}

Naturally, the previous theorem implies that one can obtain confidence intervals for the value of $\Pem(\Omega)$ when the dimension $d$ is $2,3$ or $4$. Let us fix $\alpha \in (0, 0.5)$ and let $Z_\alpha $ be the $1-\alpha$ quantile of the standard normal distribution. That is, $Z_\alpha$ is such that
$$ \mathbb{P}\left( N(0,1) \leq Z_\alpha  \right)= 1-\alpha.$$
Then provided that
$$ \frac{1}{n^{1/d}} \ll \veps_n \ll \frac{1}{n^{1/5}},  $$
it follows that with probability converging to $1-\alpha$, $  \Pem(\Omega) \in (a_n^- , a_n^+)$, where
\[  a_n^{\pm} := \frac{1}{\sigma_d}  ( \GPe{\Omega} \pm \sqrt{\frac{4 C_d \Pem(\Omega) }{n\veps_n}} Z_{\alpha/2} ). \]

 For $d \geq 5$, the confidence intervals can not be obtained unless one has some extra quantitative information on the smoothness of $\partial \Omega,$ such as curvature bounds or upper bounds on the bias. Nevertheless, since it is known that the nonlocal perimeter $\Pem_{\veps_n}(\Omega)= \E ( \GPe{\Omega})$ is less than $\sigma_d \Pem(\Omega)$ (see \eqref{PeP}), one can construct a test for the hypothesis that $\Pem(\Omega)$ is \emph{less}  than a certain number $\rho$ without using any quantitative estimates on the smoothness of $\partial \Omega$. In this case, we have automatic upper bounds for the bias and may consider the hypothesis testing of

\begin{equation}
 H_0: \Pem(\Omega) \leq \rho ,  \quad  \text{vs}  \quad  H_A: \Pem(\Omega)> \rho,
 \label{TestPerimeter}
\end{equation}
based on our estimator $\GPe{\Omega}$.
We consider the statistic:
\begin{equation}
l_n := \sqrt{\frac{n \veps_n}{4 C_d \rho}} \left( \GPe{\Omega} - \sigma_d\rho \right). 
\label{lnStatistic}
\end{equation}
The test consists on
\begin{equation}
 \text{Accept } H_0 \text{ if } l_n \leq Z_\alpha, \quad \text{reject otherwise. }
 \label{hypothesisTestGPer}  
\end{equation} 

\begin{proposition} \label{hyprop}
Assume $d \geq 2$ and  $\frac{1}{n^{1/d}}\ll \veps_n \ll 1$. Then, the type I error satisfies
\[ \limsup_{n \rightarrow \infty}  \mathbb{P}_{H_0} \left(  l_n > Z_\alpha  \right)  \leq \alpha, \]
i.e., the type I error is asymptotically below $\alpha$. The  type II error satisfies
\[ \mathbb{P}_{H_A}(l_n \leq Z_\alpha ) = O\left(\frac{1}{\sqrt{n \veps_n}} \right). \] 
\end{proposition}

\begin{remark}[Extensions]
The estimates obtained in Theorem \ref{TheoremGeneral} and Theorem \ref{th:toter} are not exclusive to the case where the points are uniformly distributed in the unit cube and to geometric graphs. In fact, with slight modifications to the proofs of Theorem \ref{TheoremGeneral} and Theorem \ref{th:toter}, we can extend these results to more general situations. For example, if the data points are distributed according to some smooth density $p$ that is supported on a regular, bounded domain $D \subset \R^d$ with $p$ bounded below and above by positive constants, then the results  still hold. In this case, the limiting value of $\GPe{\Omega}$ is a \emph{weighted} perimeter $\Pem(\Omega,p^2)$ (see for example \cite{GTS}). Convergence is guaranteed for the same scaling for $\veps_n$ as in the uniform case. The fact that the weight is $p^2$ and not $p$ (as may be a priori expected) ultimately comes from the fact that a graph cut is a double sum.

Furthermore if instead of weights $\1_{\{\| \x_i - \x_j\| \leq \veps_n\}}$ in the definition of the graph cut, one considers edge weights $\eta(\| \x_i - \x_j\| / \veps_n)$ where $\eta$ is nonnegative, integrable and non-increasing, the results still hold, provided we change the surface tension $\sigma_d$ with a surface tension associated to $\eta$ defined as in \cite{GTS}.   
\end{remark}

\begin{remark}
The previous results allow us to construct asymptotic confidence intervals for $\Pem(\Omega)$ for an arbitrary set $\Omega$ with smooth boundary and $\dist(\Omega, \partial D) >0$, using the estimator $\GP{\Omega}$. Nevertheless, since very small sets (in the sense of volume) may have very large perimeter the error estimates are not uniform in $\Omega$ and thus the asymptotic confidence intervals may not be of direct practical use. This issue is unavoidable without further assumptions on the set $\Omega$ or a change of framework for testing. One possible way to restrict the class of $\Omega$ considered is pointed out in Remark \ref{remarkUniform}. In the class of sets $\Omega$ to which Remark \ref{remarkUniform} applies, the error of approximation is uniformly controlled. 

Alternatively, one may consider the \textit{property testing} framework used in \cite{Neeman,Kothari}, where one could test the property ``the set $\Omega$ is close to a set whose perimeter is less than $\rho$" reliably. In that framework we would not be using the estimator $\GP{\Omega}$ to test whether $\Pem(\Omega)$ is below $\rho$ or not, but rather whether $\Omega$ is \textit{close} to a set whose perimeter is below $\rho$ or not. We notice that in this paper we have taken the classical hypothesis testing approach. 
 
\end{remark}

\emph{Outline.}  In Subsection \ref{sec:out} we establish \eqref{temp1} and give an outline of the argument behind our main results. 
We present the proof of Theorem \ref{TheoremGeneral} in Section \ref{sec:2}, while in Subsection \ref{sec:opt} we show that the scaling is sharp (up to logarithmic corrections) in the sense that if $n^{2}\veps^{d+1}_n \to 0$ 
then $ \GP{\Omega}$ converges in probability to zero and hence does not converge almost surely to the (rescaled) relative perimeter.   
In Section \ref{ProofTheoremDistrib} we prove the results on the asymptotic distribution of the error stated in Theorem \ref{AsympDistribFirstRegime}. In Subsection \ref{TestingSec} we study the type I and type II errors of the hypothesis test of Proposition \ref{hyprop}. Finally, Appendix \ref{AppendixLemmaBias} deals with the bias estimate from Lemma \ref{pnlp}.


\subsection{Discussion.}
\label{Discussion} Here we discuss the connections between our work and  related works in the literature. First we relate it to other estimators of perimeter based on a random sample. Then we contrast the type of the convergence and the scaling regimes  considered in this paper with the ones needed for the convergence of graph-cut based machine learning algorithms for clustering and related tasks.

The problem of estimating the perimeter of a set, $\Omega$, based on knowing which points of a random sample $V_n$ belong to $\Omega$, has been considered by a number of works. 
Cuevas, Fraiman, and Rodr{\'i}guez-Casaet \cite{CuFrRo07},
considered estimators of the Minkowksi content, which agrees with perimeter for regular enough sets, but is a less general notion of the perimeter than the one we consider, \eqref{perimeter}. Their estimator is based on counting the vertices near the boundary (relative to a parameter $\veps_n$), while we ``count" edges of a graph. The error bound obtained was of order 
$O(n^{ -1/2d })$. Pateiro-L\'opez and Rodr\'iguez-Casal \cite{PatRod08} consider a similar estimator and improve the bounds to  $O((\ln n/n)^{1/(d+1)})$ under qualitative regularity conditions (rolling ball conditions). 
Cuevas, Fraiman, and Gy{\"o}rfi  \cite{CuFrGy13}, obtain the convergence of estimators similar to those of
\cite{CuFrRo07}, when $\veps_n \gg 1/n^{1/d}$ and under weaker conditions on the regularity of $\Omega$, although still not in the full generality we consider in this paper.

Armend\'ariz, Cuevas, and Fraiman \cite{ArCuFr09} consider a similar set-up to that in \cite{CuFrRo07, PatRod08} but with different sampling rates for $\Omega$ and $\Omega^c$: let $n$ denote the number of sample points in $\Omega$ and $k$ the number of sample points in $\Omega^c$.
Under mild assumptions on the regularity of $\partial \Omega$, and under some conditions on $n$ and $k$ which include  $1 \gg \veps_{n,k} \gg (1 / n^{1/3})$ and $k \gg  (n/\veps_{n,k})^{d/2}$
they obtain the asymptotic distribution of the error, under the same scaling in $n$ and $\veps$ that we consider in Theorem \ref{AsympDistribFirstRegime}. 
We notice that our estimator is different, and also that in
\cite{ArCuFr09} a very large number of points $k$ in $\Omega^c$ is needed for the consistency to hold. This allows the authors to obtain the asymptotic distribution of the total error in any dimension, while in our setting we only obtain it in low dimensions.
A further difference between their work and ours, is that we allow for a wider range in $\veps$, namely $1 \gg \veps_n \gg (1 / n)^{2/(d+1)}$.

Jim{\'e}nez and Yukich \cite{JinYuk11} give a different use to the point cloud and instead of considering a parameter $\veps$ to count points close to the boundary of the set or to define a geometric graph, they consider a new estimator based on the Delaunay triangulation induced by the cloud. They obtain results not only on estimating the perimeter of the set, but also integrals of functions over $\partial \Omega$. 

Kothari, Nayyeri, and O'Donnell \cite{Kothari} and 
 Neeman \cite{ Neeman} consider an estimator essentially based on the following procedure: pick $n$ random points uniformly distributed on $D$  and to each point associate a random direction for a `needle' based at the point with length of order $\sqrt{\veps_n}$; then count how many of the needles touch the boundary of $ \Omega$. Their main motivation is to consider the perimeter estimation from the viewpoint of \textit{property testing} as introduced in \cite{TestingProperties}. In that setting, the idea is to produce an algorithm that requires a small number of samples (essentially independent of the dimension $d$) in order to determine if a given set has a small perimeter or is `far away' from a set that has small perimeter.  
The authors show completeness and soundness of the test they design (the notion of completeness and soundness is as in \cite{TestingProperties, Kothari}). The notion of testing used in their work is one of the main differences with our work since we consider the perimeter testing in the more classical framework of hypothesis testing (Proposition \ref{hyprop})\nc. We note that the completeness of \cite{Kothari} is analogous to the type I error, but the soundness is fundamentally different from estimating the type II error.
 %

It is also worth mentioning the work of Belkin, Narayanan and Niyogi \cite{ConvexSets} where they consider an algorithm that requires as few samples as possible in order to estimate the perimeter of a convex body. Their results show that there is an algorithm that uses $O(d^4 \gamma^{-2} )$ samples to obtain an estimator for the true perimeter of the convex set, with an error of approximation of $\gamma$; this statement holds with high probability. We notice the polynomial dependence on dimension in their estimates.

Let us now contrast the type of the convergence and the scaling regimes that we consider in this paper with the ones needed for the convergence of graph-cut based machine learning algorithms for clustering and related tasks. 

Graph-cut based algorithms for tasks such as clustering have played an important role in machine learning
\cite{ACPP,  BLUV12, HagKah, HeinSetz, KanVemVet04, ShiMalik,  SzlamBresson,  vonLux_tutorial,  WeiChe89}. Data clustering algorithms are called consistent if as the sample size $n \to \infty$
their outputs converge to a desired partitioning of the underlying measure being sampled. 
It is of interest to understand under what scaling of $\veps_n$ on $n$ does the consistency hold. Here we showed that for a \emph{fixed} set $\Omega$ the value of graph perimeter
(and indeed of the graph-cut-based objective functionals such as Cheeger, ratio, or normalized cuts) converges to the perimeter of $\Omega$; this result holds even for rather sparse graphs (Theorem \ref{th:toter}). In particular, partitioning such sparse graphs (as on Figure \ref{fig:cutsp}) does not provide almost  any information about the clusters present in the data. We conclude that the convergence of the graph perimeter of a fixed set towards the continuum perimeter does not provide the needed information on the asymptotic properties of graph-cut based clustering algorithms. To obtain consistency of such algorithms one needs a stronger notion of convergence of graph based functionals towards continuum functionals.  
Recently the authors \cite{GTS}, and together with Laurent and Bresson \cite{TSvBLX_jmlr} have developed the appropriate notion of convergence (based on $\Gamma$-convergence from the calculus of variations), and have applied it to consistency of Cheeger, ratio, sparsest, and normalized cut based point cloud clustering.
More specifically, in \cite{TSvBLX_jmlr} it is shown that consistency holds if $ \frac{(\log(n))^{1/d}}{n^{1/d}} \ll \veps \ll 1$ (for $d \geq 3$), in that regime the graphs are connected with high probability as $n \to \infty$. 
\nc

\subsection{Outline of the argument} \label{sec:out}

In order to understand the asymptotic behavior of the graph perimeter, following \cite{ACPP} we first define a symmetric kernel $\phi_{\veps}: \D \times \D \rightarrow (0,\infty)$ by
$$\phi_{\veps_n}(x,y) = \frac{ \1_{\{||x-y|| \leq \veps_n\}}  }{\veps^{d+1}_n}| \1_\Omega(x) - \1_\Omega(y)|.$$ 
Using the kernel $\phi_{\veps_n}$, we can then write $\GPe{\Omega}$ as
\begin{equation}
 \GPe{\Omega}  =  \frac{2}{n (n-1)}\sum_{i=1}^{n}\sum_{j=i+1}^{n} \phi_{\veps_n}(\x_i,\x_j), 
 \label{eq:tri_tv}
\end{equation}
which is a  $\emph{U-statistic}$ in the terminology of 
\cite{hoeffding}. A simple computation shows that the mean of this $U$-statistic, is the non-local perimeter $\NLTVenon{\Omega}$ defined in \eqref{NonLocalPerimeter}, that is, 
\begin{equation}
\E(\GPe{\Omega}) =  \NLTVe{\Omega} =  \int_{\dom}\int_{\dom} \phi_{\veps}(x,y)\; \Rd x \Rd y.
\label{eq:tri_mean}
\end{equation}
\nc
Additionally, Remark 4.3 in \cite{GTS} establishes that the non-local perimeter $\NLTVe{\Omega}$ approaches a constant multiple of the relative perimeter of $\Omega$ as the parameter $\veps_n$ goes to zero. More precisely, if $\veps_n \to 0$ as $n \to \infty$ then
\begin{equation}
  \NLTVen{\Omega} \to \sigma_{d} \BV{\Omega} \qquad \te{as } n \to \infty
 \label{NonLocalpoint-wiseConv}
\end{equation}
for $\sigma_{d}$ the surface tension \eqref{sigma}. This convergence also follows from the estimates in Appendix \ref{AppendixLemmaBias} in the special case that $\Omega$ has a smooth boundary. Combining \eqref{NonLocalpoint-wiseConv} with \eqref{eq:tri_mean} we conclude that if $\veps_n$ converges to zero as $n \to \infty$ then
\begin{equation}
\E(\GP{\Omega}) \to  \sigma_{d} \BV{\Omega} \qquad \te{as } n \to \infty.
\label{MeanConvergingToPerimeter}
\end{equation}
Since the graph perimeter $\GPe{\Omega}$ is a $U$-statistic of order two we can use the general theory of $U$-statistics to obtain moment estimates for $\GPe{\Omega}$. Let us first note that Hoeffding's decomposition theorem for $U$-statistics of order two (see \cite{TheoryUStat}) implies that $\GPe{\Omega}$ can be written as:
\begin{equation}
 \GPem_{n,\veps_n}(\Omega) - \Pem_{\veps_n}(\Omega) = 2 U_{n,1} + U_{n,2}, 
 \label{CanonicalDecomp}
\end{equation}
where $U_{n,1}$ is a U-statistic of order one ( just a sum of  centered independent random variables) and $U_{n,2}$ is a U-statistic of order two which is \emph{canonical} or completely degenerate (see \cite{TheoryUStat}). In order to define the variables $U_{n,1}$ and $U_{n,2}$, let us introduce the functions
\begin{alignat}{2}
 \bar{\phi}_{\veps_n}(x) &:= \int_{\D}\phi_{\veps_n}(x,z) \, \Rd z , & \quad x &\in \D, \notag\\
 g_{n,1}(x)&:= \bar{\phi}_{\veps_n}(x) - \Pem_{\veps_n}(\Omega) , & \quad x & \in \D,   \label{CanonicalDecomp1} \\
 g_{n,2}(x,y) &:= \phi_{\veps_n}(x,y) - \bar{\phi}_{\veps_n}(x) - \bar{\phi}_{\veps_n}(y) + \Pem_{\veps_n}(\Omega), & \quad x,y &\in \D.   \notag
\end{alignat}
With the previous definitions, we can now define
\begin{align}
\begin{split}
U_{n,1} &= \frac{1}{n}\sum_{i=1}^{n} g_{n,1}(\x_i),   \\
U_{n,2} &= \frac{2}{n(n-1)}\sum_{1\leq i < j \leq n} g_{n,2}(\x_{i},\x_{j}). 
\end{split}
\label{CanonicalDecomp2}
\end{align}

%
We remark that $\int_{\D}g_{n,1}(z) \Rd z =0$ and that $\int_{\D}g_{n,2}(x,z)\Rd z  =0$ for all $x\in \D$. Because of this, $U_{n,1}$ and $U_{n,2}$ are said to be canonical statistics of order one and two respectively (see \cite{TheoryUStat}). Now, Bernstein's inequality \cite{bern24} implies that
\begin{equation}
 \E(|U_{n,1}|^p) \leq \frac{C_p}{n^p} \max \left(  A_{n,1}^p , B_{n,1}^p \right),
 \label{MomentsU1}
\end{equation}
where
\begin{equation}
A_{n,1} := || g_{n,1}||_{\infty}, \quad B_{n,1} := \sqrt{n}|| g_{n,1}||_{2}. 
\label{ConstantsU1}
\end{equation}
and $C_p$ is a universal constant. See also \cite{GineLatala} for a slight generalization of the previous result.

On the other hand some of the moment estimates in \cite{GineLatala} for canonical $U$-statistics of order two can be used to prove that
\begin{equation}
\E\left(| U_{n,2}|^p \right) \leq  \frac{C_p}{n^{2p}} \max\left(   A_{n,2}^p  , B_{n,2}^p   , C_{n,2}^p  \right),
\label{MomentsUn2}
\end{equation} 
where
\begin{equation}
A_{n,2}:= ||g_{n,2}||_{\infty}, \quad B_{n,2} := n || g_{n,2}||_2, \quad (C_{n,2})^2:= n || \int_{\D}g_{n,2}^2(\cdot,y )\; \Rd y   ||_\infty. 
\label{ConstantsU2}
\end{equation}
and $C_p$ is a universal constant. From the decomposition \eqref{CanonicalDecomp} it follows that for $p\geq1$
$$ \E \left( | \GPe{\Omega} - \Pem_{\veps_n}(\Omega) |^p \right) \leq  C_p ( \E\left( |U_{n,1}|^p \right) + \E\left( |U_{n,2}|^p \right)  ). $$ 
Thus in order to obtain the moment estimates for $\GPe{\Omega}$ in Theorem \ref{TheoremGeneral}, 
we focus on finding estimates for the quantities in \eqref{ConstantsU1} and \eqref{ConstantsU2}.  

\begin{remark}
The estimates on $U_{n,1}$ and $U_{n,2}$ exhibit a crossover in the nature when the parameter $\veps_n$ transitions between the sparse ($\frac{1}{n^{2/(d+1)}} \ll \veps_n \ll \frac{1}{n^{1/d}} $ ) and the dense regime ( $\frac{1}{n^{1/d}} \ll \veps_n \ll 1  $).  From the theory of $U$-statistics the crossover in the nature of the bounds, is connected to the different nature of the two components in the canonical decomposition for $U$-statistics. Under the dense regime the biggest source of error comes from the term $U_{n,1}$ while in the sparse regime the biggest source of error comes from the term $U_{n,2}$. The two variables $U_{n,1}$ and $U_{n,2}$ exhibit a different nature. In fact, one can think that $U_{n,1}$ is a global quantity as it is the sum of averages, while $U_{n,2}$ is simply the sum of pure interactions concentrating on the boundary of the set $\Omega$. We believe that there is a deeper geometric and analytic reason for the scaling of the error that appears in the sparse regime ($ \frac{1}{n^{2/(d+1)}} \ll\veps_n \ll \frac{1}{n^{1/d}} $). Moreover, we believe that such understanding would allow us to complete the  study of the asymptotic distribution in Theorem \ref{AsympDistribFirstRegime} for the sparse regime. We expect such distribution to be of the Gaussian chaos type.
\end{remark}
\medskip

The bias estimates in Appendix \ref{AppendixLemmaBias} are obtained by a series of computations whose starting point is writing $\Pem_\veps(\Omega)$ in terms of an iterated integral, the outer one taken over the manifold $\partial \Omega$ and the inner one taken along the normal line to $\partial \Omega$ at an arbitrary point $x \in \partial \Omega$. Such computations show that the first order term of $\Pem_{\veps_n}(\Omega)$ on $\veps_n$ vanishes. 

Finally, Theorem \ref{AsympDistribFirstRegime} is obtained by using the canonical decomposition of $U$-statistics and by noticing that in the dense regime $\frac{1}{n^{1/d}}\ll \veps_n \ll 1$, the variable $U_{n,2}$ is negligible in relation to $U_{n,1}$. We make use of the CLT for triangular arrays after computing the variances of the involved variables.

\section{Proof of Theorem \ref{TheoremGeneral}} \label{sec:2}
We first compute the moments of $U_{n,1}$ and so we start  computing the quantities $A_{n,1}$ and $B_{n,1}$ from \eqref{ConstantsU1}. Denote by $T_{\veps}$ the $\veps$-\emph{tube} around $\partial \Omega$, that is,  consider the set
 \begin{equation}
  T_{\veps}:= \left\{ x\in \R^d  \: : \: \dist(x,\partial \Omega) \leq \veps \right\}.  
  \label{tube}
  \end{equation}
We also consider the \emph{half tubes} $\tubi$ and $\tubo$,
\begin{equation}
 T_{\veps}^-:=  \left\{ x\in \Omega  \: : \: \dist(x,\partial \Omega) \leq \veps \right\}, \quad  T_\veps^+:= \left\{ x\in \Omega^c  \: : \: \dist(x,\partial \Omega) \leq \veps \right\}.
 \label{halfTube}
\end{equation}
With these definitions it is straightforward to check that  
\begin{equation}
\bar{\phi}_{\veps_n}(x) =
  \begin{cases} 
      \hfill |  B_d(x,\veps_n) \cap \Omega |/ \veps_n^{d+1}    \hfill & \text{ if $x \in T_{\veps_n}^+$ } \\
      \hfill | B_d(x,\veps_n) \cap \Omega^c |/ \veps_n^{d+1}    \hfill & \text{ if $x \in T_{\veps_n}^-$ } \\
      \hfill 0 \hfill & \text{ if $x \not \in T_{\veps_n}$ }. \\
  \end{cases}
  \label{phiAverageDef}
\end{equation}    
Since $| B_d(x,\veps_n) \cap \Omega|$ and $|  B_d(x,\veps_n)\cap \Omega^c | $ are bounded by $\alpha_d \veps_n^d$, where $\alpha_d$ is the volume of the $d$-dimensional unit ball, we deduce that 
$$ A_{n,1} = O\left(\frac{1}{\veps_n}  \right). $$ 

In order to compute the quantity $B_{n,1}$ we use the following lemma, whose proof may be found in Appendix \ref{AppendixMomentsPhi}.
\begin{lemma}
Let $p \geq 1$ and let $\Omega  \subseteq \D$, be a set with finite perimeter. Then, for all $\veps>0$ we have
$$  \int_{\D}\bar{\phi}_{\veps_n}^p(x) \Rd x \leq \frac{\alpha_d^{p-1} \sigma_d}{\veps^{p-1}} \Pem(\Omega).   $$
In particular, taking $p=1$ in the previous expression, we obtain
\begin{equation} \label{PeP} 
 \Pem_{\veps}(\Omega) \leq \sigma_d \Pem(\Omega). 
\end{equation}
\label{LemmaMomentsAveragephi}
\end{lemma}

Using the previous lemma with $p=2$ we deduce that $ \int_{\D}\bar{\phi}_{\veps_n}^2(x) \Rd x = O\left(\frac{1}{\veps_n} \right)  $, and since
$$ \int_{\D} g_{n,1}^2(x) \Rd x = \int_{\D}\bar{\phi}_{\veps_n}^2(x) \Rd x - \left(\Pem_{\veps_n}(\Omega) \right)^2,  $$ 
we conclude that 
$$B_{n,1}= O \left( \sqrt{\frac{n}{\veps_n}}   \right).  $$
From the previous computations, we deduce that
$$  \E(|U_{n,1}|^p) \leq C_{p,d} \max\{ 1 , \BV{\Omega} \}^{p} \nc \max \left(  \frac{1}{n^p\veps_n^p, } , \frac{1}{n^{p/2} \veps_n^{p/2} }  \right), $$
where \nc$C_{p,d}$ depends on $p$ and $d$, but is independent of the set $\Omega$. If $\frac{1}{n^{2/(d+1)}} \leq \veps_n$, so that in particular $\frac{1}{n \veps_n}$ is $o(1)$, then  
\begin{equation}
\E(|U_{n,1}|^p) \leq  \frac{C_{p,d} \max\{ 1 , \BV{\Omega} \}^{p} \nc}{n^{p/2} \veps_n^{p/2} }  .   
\label{MomentsUn1Final}
\end{equation}

\begin{remark}
 Later on, in Lemma \ref{LemmaVarianceg}, we provide an explicit computation of $\int_{\D}\bar{\phi}_{\veps_n}^2(x) \Rd x $ up to order $\frac{1}{\veps_n}$ , which is useful when studying the asymptotic distribution of a rescaled version of $U_{n,1}$.
\end{remark}

Now we turn to the task of obtaining moment estimates for $U_{n,2}$. We estimate the quantities $A_{n,2}$, $B_{n,2}$ and $C_{n,2}$ from \eqref{ConstantsU2}. Let us start by estimating $A_{n,2}$. Note that for any $(x,y) \in \D \times \D$,  $\bar{\phi}_{\veps_n}(x)$ and $\bar{\phi}_{\veps_n}(y)$ are of order $\frac{1}{\veps_n}$  and that $\Pem_{\veps_n}(\Omega)$ is of order one. Thus, it is clear from the definition of $g_{n,2}$ in \eqref{CanonicalDecomp1} that
$$ A_{n,2} = O\left(\frac{1}{\veps_n^{d+1}}\right).  $$ 
On the other hand, using $\phi_{\veps_n}^2(x,y) = \frac{1}{\veps_n^{d+1}} \phi_{\veps_n}(x,y)$, we obtain that for every $x \in \D$, 
%
%
\begin{align}  
\begin{split}
\int_{\D}g_{n,2}^2(x,y) \Rd y = & \int_{\D} \phi_{\veps_n}^2(x,y) \Rd y - \bar{\phi}_{\veps_n}^2(x) + 2 \theta_n \overline{\phi}_{\veps_n}(x) \\
 & -  2 \! \int_{\D}\phi_{\veps_n}(x,y)\bar{\phi}_{\veps_n}(y) \Rd y + \! \int_{\D} \bar{\phi}_{\veps_n}^2(y) \Rd y - \theta_n^2 \\
 = & \frac{1}{\veps_n^{d+1}} \bar{\phi}_{\veps_n}(x) - \bar{\phi}_{\veps_n}^2(x) + 2 \theta_n \overline{\phi}_{\veps_n}(x) \\
&   - 2 \int_{\D}\phi_{\veps_n}(x,y)\bar{\phi}_{\veps_n}(y) \Rd y + \int_{\D} \bar{\phi}_{\veps_n}^2(y) \Rd y -\theta_n^2, 
\label{aux1}
\end{split}
 \end{align}
 where we are using $\theta_n := \Pem_{\veps_n}(\Omega)$. From this, it follows that
$$ C_{n,2}=O \left( \sqrt{ \frac{n}{\veps_n^{d+2}} }\right).   $$
Finally, upon integration of \eqref{aux1} and direct computations, we obtain
$$ || g_{n,2}||_{2}^2 = \frac{\theta_n}{\veps_n^{d+1}} - 2 \int_{\D}\bar{\phi}_{\veps_n}^2(y)\Rd y + \theta_n^2,    $$
which implies that
$$ B_{n,2} = O \left(\frac{n}{\veps_n^{(d+1)/2}}   \right). $$
Thus, from \eqref{MomentsUn2} we deduce that
$$ \E(|U_{n,2}|^p) \leq K_{p,d} \max\left( \frac{1}{n^{2p}\veps_n^{p(d+1)}}  ,  \frac{1}{n^p\veps_n^{p(d+1)/2}}, \frac{1}{n^{3p/2}\veps_n^{p(d+2)/2}}  \right), $$
where $K_{p,d} = C_{p,d} ( \max\{1,\BV{\Omega} \} )^{p}$ for $C_{p,d}$ some constant that does not depend on the set $\Omega$\nc. Hence, if $\frac{1}{n^{2/(d+1)}} \leq \veps_n$, we have
\begin{equation}
  \E(|U_{n,2}|^p) \leq  \frac{K_{p,d}}{n^p\veps_n^{p(d+1)/2}}.
  \label{MomentsUn2Final}
\end{equation}
Combining \eqref{MomentsUn1Final} and \eqref{MomentsUn2Final} and using the canonical decomposition \eqref{CanonicalDecomp}, we obtain \eqref{MomentsTheorem}.

 \subsection{Sharpness of the Rate in Theorem \ref{TheoremGeneral}} \label{sec:opt}
A very simple argument shows that the rates for $\veps_n$ that guarantee the almost sure convergence of the graph perimeter to the actual perimeter in Theorem \ref{TheoremGeneral} are optimal in terms of scaling, up to logarithmic corrections. 

In fact, suppose $n^{2}\veps^{d+1}_n = o(1)$ and let $e_{n}$ denote the random variable that counts the number of edges that cross in the interface between $\Omega$ and its complement. In other words, we define
$$
e_{n} :=  \veps_n^{d+1}\sum^{n}_{i=1} \sum^{n}_{j=i+1} \phi_{\veps_n}(\x_i,\x_j).
$$
As a consequence, if $\Omega$ has finite perimeter then we have
\begin{equation}
\mathrm{GPer}_{n,\veps_n}(\Omega) = \frac{2}{n (n-1) \veps^{d+1}_n} e_n, \qquad \E(e_n) = \frac{ n (n-1) \veps^{d+1}_n }{2}  \Pem_{\veps_n}(\Omega).
\label{Expecesubn}
\end{equation}

Note that $e_n$ takes integer values in the range $\{0,1,\ldots,N\}$ for $N = n(n-1)/2,$ so that
$$
\E(e_n) = \sum^{N}_{k=1} \; k p^{n}_{k} \qquad p^{n}_{k} := \Prob( e_{n} = k ).
$$
The fact that $p^{n}_{0} + \cdots + p^{n}_{N} = 1$ implies
$$
\E(e_n) =  \sum^{N}_{k=1} kp^{n}_{k} \geq  \sum^{N}_{k=1} p^{n}_{k}  =(1-p^{n}_{0}).
$$
In particular, from \eqref{Expecesubn} and \eqref{NonLocalpoint-wiseConv} we deduce that if $n^{2}\veps^{d+1}_{n} \to 0$ and $\Omega$ has finite perimeter then
$$
(1-p^{n}_{0}) \leq \E(e_{n}) = o(1).
$$
On the other hand, note that for any given $\gamma>0$ it is true that $\GPem_{n,\veps_n}(\Omega)  > \gamma $ implies that $e_n \not =0$. In turn
$$ \Prob \left( \GPem_{n,\veps_n}(\Omega)  > \gamma   \right) \leq \Prob  \left( e_n\not =0  \right)  = 1- p_0^n =o(1).  $$
We conclude that if $n^{2}\veps^{d+1}_{n} \to 0$ then $\GPem_{n,\veps_n}(\Omega) $ converges in probability to zero. Therefore, if $\Omega$ has a non-zero, finite perimeter then $\GPem_{n,\veps_n}(\Omega) $ does not converge to $\sigma_{d} \Pem(\Omega)$ in probability (nor almost surely, either).
%
%
%

\section{Proof of Theorem \ref{AsympDistribFirstRegime}}
\label{ProofTheoremDistrib}

The proof of Theorem \ref{AsympDistribFirstRegime} relies on the following lemma, whose proof may be found in Appendix \ref{AppendixVarianceg}.
\begin{lemma}
Asssume that $\veps_n \rightarrow 0$ as $n \rightarrow \infty$, and that $\Omega \subset D$ is an open set with smooth boundary so that $ \mathrm{dist}(\Omega,\partial D) > 0$. Then
\begin{equation}
\Var(g_{n,1}(X_1))= \frac{C_d \Pem(\Omega)}{\veps_n} + O(1),
\label{Varianceg1}
\end{equation}
where $C_d$ is given by \eqref{Cd}.
\label{LemmaVarianceg}
\end{lemma}

Now we turn our  attention to the proof of Theorem \ref{AsympDistribFirstRegime}.

\begin{proof}[Proof of Theorem \ref{AsympDistribFirstRegime}]
Note that from \eqref{CanonicalDecomp}, \eqref{CanonicalDecomp1} and \eqref{CanonicalDecomp2} we obtain
\begin{align*}
 \sqrt{ \frac{n \veps_n }{4 C_d \Pem(\Omega)} } \left( \GPe{\Omega} - \Pem_{\veps_n}(\Omega)  \right) = & \sqrt{ \frac{ \veps_n }{n C_d \Pem(\Omega)} } \sum_{i=1}^{n} g_{n,1}(\x_i)  \\
 & + \sqrt{ \frac{n \veps_n }{4 C_d \Pem(\Omega)} }U_{n,2}.     
\end{align*}
From the moment estimates \eqref{MomentsUn2Final}, we deduce that
\begin{equation} 
\sqrt{ \frac{n \veps_n }{4 C_d \Pem(\Omega)} } U_{n,2} \overset{P}{\rightarrow} 0. 
\label{SecondOrderStatisticVanishes}
\end{equation}
On the other hand, we note that from \eqref{Varianceg1} 
$$  \frac{\error(g_{n,1}(\x_1)) \sqrt{\veps_n} }{\sqrt{C_d \Pem(\Omega)}} \rightarrow 1 , \quad \text{as } n \rightarrow \infty, $$
where $\error(g_{n,1}(\x_1)) $ is the standard deviation of $g_{n,1}(\x_1)$. Lyapunov's condition which is sufficient to allow us to use the central limit theorem for triangular arrays  is easily checked from  Lemma \ref{LemmaMomentsAveragephi}. We deduce that
$$  \sqrt{ \frac{ \veps_n }{n C_d \Pem(\Omega)} } \sum_{i=1}^{n} g_{n,1}(X_i) \overset{w}{\longrightarrow}N(0,1). $$  
Combining with \eqref{SecondOrderStatisticVanishes} and the Slutsky's theorem, we obtain the desired result.
Finally, to obtain the last statement in the theorem, we note that from the bias estimates in Lemma \ref{pnlp},
$$  \sqrt{n \veps_n } | \Pem_{\veps_n}(A) - \sigma_d \Pem(A) |  = O(n^{1/2}\veps_n^{5/2})  $$
Under the condition \eqref{conditionBiased}, we conclude that $ \sqrt{n \veps_n } | \Pem_{\veps_n}(A) - \sigma_\eta \Pem(A) | \rightarrow 0$. This implies \eqref{eqn2}.
\end{proof}

\subsection{Application to Perimeter Testing} \label{TestingSec}
Here we prove Proposition \ref{hyprop}.
We assume that $\Omega \subseteq \D$ is an open set with smooth boundary such that $\overline \Omega \subset \D$. Note that under the null hypothesis,  if $l_n > Z_\alpha$, then, 
\begin{equation}
 Z_\alpha <  l_n \leq  \sqrt{\frac{n \veps_n}{4 C_d \Pem(\Omega)}} \left(   \GPe{\Omega} -  \Pem_{\veps_n}(\Omega)  \right),
\end{equation}
where we used that $\Pem_{\veps_n}(\Omega) \leq \sigma_d \Pem(\Omega)$
by Lemma \ref{LemmaMomentsAveragephi}. Thus, using Theorem \ref{AsympDistribFirstRegime}, we deduce that asymptotically, the type I error of our test is
\begin{align*}
  \mathbb{P}_{H_0} \left(  l_n > Z_\alpha  \right) \leq \, & \mathbb{P}\left(  \sqrt{\frac{n \veps_n}{4 C_d \Pem(\Omega)}} \left(   \GPe{\Omega} -  \Pem_{\veps_n}(\Omega)  \right) > Z_\alpha  \right) \\
  & \longrightarrow  \mathbb{P}\left(  N(0,1) > Z_\alpha \right) = \alpha,  
\end{align*}
which establishes the first part of Proposition \ref{hyprop}. 
In order to compute the type II error of our test, suppose that $\Pem(\Omega)=\rho'$ where $\rho' > \rho$. In that case,
\begin{align}
\begin{split}
\mathbb{P}_{H_A}&\left( l_n \leq Z_\alpha \right)=  \mathbb{P}_{H_A}
\left( \sqrt{\frac{n \veps_n}{4 C_d \rho}} (\sigma_d \rho -
\Pem_{n,\veps_n}(\Omega)  ) \geq - Z_\alpha   \right) \\
&= \mathbb{P}_{H_A} \left( \sqrt{\frac{n \veps_n}{4 C_d \rho}}
(\Pem_{\veps_n}(\Omega) - \Pem_{n,\veps_n}(\Omega)  ) \geq  - Z_\alpha   +
\sqrt{\frac{n \veps_n}{4 C_d \rho}}( \Pem_{\veps_n}(\Omega) - \sigma_d
\rho )  \right),
\end{split}
\end{align}
Now recall that $\lim_{n \rightarrow \infty}\Pem_{\veps_n}(\Omega)=  \sigma_d \Pem(\Omega) = \sigma_d \rho' > \sigma_d \rho $. In particular, we deduce that
$$ \lim_{n \rightarrow  \infty} \sqrt{\frac{n \veps_n}{4 C_d \rho}}(\Pem_{\veps_n}(\Omega) - \sigma_d \rho ) = +\infty.    $$
Thus, for large enough $n$,
$$ - Z_\alpha+\sqrt{\frac{n \veps_n}{4 C_d \rho}}( \Pem_{\veps_n}(\Omega)
- \sigma_d \rho ) \geq \frac{1}{2} \sqrt{\frac{n \veps_n}{4 C_d \rho}}(\sigma_d \rho' - \sigma_d \rho ). $$
Hence, for large enough $n$,
$$ \mathbb{P}_{H_A}\left( l_n \leq Z_\alpha \right) \leq \mathbb{P}_{H_A} \left( \sqrt{\frac{n \veps_n}{4 C_d \rho}} (\Pem_{\veps_n}(\Omega) - \Pem_{n,\veps_n}(\Omega)  ) \geq \frac{1}{2} \sqrt{\frac{n \veps_n}{4 C_d \rho}}( \sigma_d \rho' - \sigma_d \rho ).  \right)     $$
Using the moment estimates from Theorem \ref{TheoremGeneral}, and Markov's inequality, we deduce that
$$ \mathbb{P}_{H_A}(l_n \leq Z_\alpha ) = O\left(\frac{1}{\sqrt{n \veps_n}} \right). $$
That is, the type II error is of order $\frac{1}{\sqrt{n \veps_n}}$.

\appendix

\section{Proof of Lemma \ref{pnlp} }
\label{AppendixLemmaBias}
Since $\Omega \subset \subset \dom$ and $\Omega$ has smooth boundary the relative perimeter of $\Omega$ with respect to $\dom$ in the generalized sense \eqref{perimeter} simply corresponds to the usual perimeter of $\partial \Omega$ in the sense that
$$  \Pem(\Omega) =  \int_{\partial \Omega } \; \Rd \mathcal H^{d-1} = \mathcal H^{d-1}(\partial \Omega).$$
Additionally, for all $\veps \leq \delta := \dist(\Omega , \partial \D)$  we have that  
$$\NLTVenon{\Omega}= \frac{2}{\veps^{d+1}}\int_{\Omega} | B_{d}(x,\veps) \cap \Omega^c| \; \Rd x,$$
where $B_{d}(x,r)$ denotes the ball of radius $r$ in $\mathbb{R}^{d}$ centered at $x$ and $\Omega^c$ denotes the complement of $\Omega$ in all of space. Moreover, since $\partial \Omega$ is a compact smooth manifold, we can assume without the loss of generality ( by taking $\veps $ small enough) that for every $x \in T_{\veps}$ there is a unique point $P(x)$ in $\partial \Omega$ closest to $x$. Furthermore, we can assume that the map $P$ is smooth. We may further write 
%

$$ \NLTVenon{\Omega} = \frac{2}{\veps^{d+1}} \int_{\tubi}| B_{d}(x,\veps) \cap \Omega^c| \; \Rd x, $$
where $\tubi$ is defined in \eqref{halfTube}. This reformulation makes it natural to write the previous integral as an iterated integral; the outer integral is taken over the manifold $\partial \Omega$ and the inner integral is taken along the normal line to $\partial \Omega$ at an arbitrary point $x$ along the boundary.

To make this idea precise, we first let $\textbf{N}(x)$ denote the outer unit normal to $\partial \Omega$ at $x \in \partial \Omega$ and then consider the transformation $(x,t) \in \partial \Omega \times (0, 1) \mapsto x - t \veps \textbf{N}(x)$ for all $\veps$ sufficiently small. The Jacobian of this transformation equals $\veps \det(I + t \veps \textbf{S}_x)$, where $\textbf{S}_x$ denotes the shape operator (or second fundamental form) of $\partial \Omega$ at $x,$ see \cite{Tubes} for instance. For all $\veps$ sufficiently small, we may therefore conclude that
\begin{equation*}
\frac1{\veps}  \int_{\tubi} \!\! | B_{d}(x,\veps) \cap \Omega^c| \; \Rd x =  \int_{\partial \Omega} \!\!\left(\int_{0}^{1} |B_{d}(x-t\veps \textbf{N}(x),\veps) \cap \Omega^c| \det(I + t\veps\textbf{S}_x) \Rd t\right) \Rd \mathcal{H}^{d-1}(x).
\end{equation*}
As a consequence, we also have that
\begin{equation} \label{perint}
  \NLTVenon{\Omega}  =  \frac{2}{\veps^{d}} \int_{\partial \Omega}\left(\int_{0}^{1} |B_{d}(x- t \veps \textbf{N}(x),\veps) \cap \Omega^c| \det(I + t \veps \textbf{S}_x) \; \Rd t\right) \; \Rd \mathcal{H}^{d-1}(x).  
\end{equation}
With the expression \eqref{perint} in hand, we may now proceed to establish \eqref{errorNonlocal} by expanding $\Pem_\veps(\Omega)$ in terms of $\veps$ and appealing to some elementary computations that show that the first order term in $\veps$ vanishes.

For a fixed $x \in \partial \Omega,$ we first wish to understand the behavior of the function
$$  g_{x}(\veps) := \frac{1}{\veps^d} \left(\int_{0}^{1} |B_{d}(x- t \veps \textbf{N}(x),\veps) \cap \Omega^c| \det(I + t \veps \textbf{S}_{x}) \; \Rd t\right)$$
for $\veps$ in a neighborhood of zero. Without loss of generality, we may assume that $x= 0$, that $\textbf{N}(x)=e_d$ and that around  $x$ the boundary $\partial \Omega$ coincides with the graph $\hat x = (x_1, \dots, x_{d-1}) \mapsto (\hat{x}, f(\hat{x})) \in \mathbb{R}^{d}$  of a smooth function $f(\hat x)$ that satisfies both $f(0)=0$ and $\nabla f (0) =0$ simultaneously. By symmetry of the shape operator $\textbf{S}_{x},$ there exists an orthonormal basis for $\R^{d-1}$ (where we identify $\R^{d-1}$ with the hyperplane $\left\{ (\hat{x},x_d) \: : \: x_d =0 \right\}$) consisting of eigenvectors of the shape operator. We let $v_1, \dots, v_{d-1} $ denote the eigenvectors of $\textbf{S}_{x}$ and $\kappa_1, \dots, \kappa_{d-1}$ the corresponding eigenvalues ( also known as principal curvatures). In particular, whenever $\|\hat{x}\| \leq \veps$ we have that
\begin{equation}
 f(\hat{x}) = \frac{1}{2} \sum_{i=1}^{d-1} \kappa_i \langle \hat x , v_i \rangle ^2  + O(\veps^3)
 \label{SecOrderf}
\end{equation}
where curvatures $\kappa_{i} = \kappa_{i}(x)$ and the $O(\veps^3)$ error term can be uniformly bounded.

With these reductions in place, we first define $u(\hat y) := \sqrt{\veps^2 - \|\hat y\|^2} $ and then let
$$
h(\hat y,t;\veps) := 
\begin{cases}
2u(\hat y) & \text{if} \quad f(\hat y) + \veps t < -u(\hat y) \\
u(\hat y) - \veps t - f(\hat y) & \text{if} \quad -u(\hat y) \leq  f(\hat y) + \veps t  \leq u(\hat y) \\
0 & \text{otherwise}.
\end{cases}
$$
A direct calculation then shows that
\begin{align}\label{eq:this_formula}
|B_{d}(x- t \veps \textbf{N}(x),\veps) \cap \Omega^c| = \int_{B_{d-1}(0,\veps)} h(\hat y,t;\veps) \; \Rd \hat y,
\end{align}
and an application of \eqref{SecOrderf} shows that $h(\hat y,t;\veps) = 2\sqrt{\veps^2 - \|\hat y\|^2}$ only if
$$
\| \hat y\|^2 = \veps^{2} - O(\veps^{4}) \qquad \text{and} \qquad u(\hat y) = O(\veps^2).
$$
It therefore follows that
\begin{align*}
\int_{B_{d-1}(0,\veps) \cap \{ f(\hat y) + \veps t < -u(\hat y) \} }  \!\!\! h(\hat y,t;\veps) \; \Rd \hat y &\leq O(\veps^2) \int_{B_{d-1}(0,\veps) \cap \{ \|\hat y\| \geq \sqrt{\veps^2 - O( \veps^4)} \} } \!\! \Rd \hat y = O(\veps^{d+3}).
\end{align*}
We then let $A^{\veps}_{t}$ denote the set $A^{\veps}_{t} := \{ \hat y \in B_{d-1}(0,\veps): -u(\hat y) \leq  f(\hat y) + \veps t  \leq u(\hat y) \}$ and use the previous estimate in \eqref{eq:this_formula} to uncover
\begin{equation}\label{eq:first_bdry_est}
|B_{d}(x- t \veps \textbf{N}(x),\veps) \cap \Omega^c| = \int_{B_{d-1}(0,\veps) \cap A^{\veps}_{t} } ( u(\hat y) - \veps t - f(\hat y)  )\; \Rd \hat y + O(\veps^{d+3}). 
\end{equation}
We may then note that
$$  \det(I + \veps t \textbf{S}_{x}) = (1+ t \veps \kappa_1)  \dots (1 + t \veps \kappa_{d-1}) = 1 + t \veps H_{x}  + O(\veps^2),   $$
where $H_{x} :=\sum_{i=1}^{d-1}\kappa_i $ represents the mean curvature. Using this fact in \eqref{eq:first_bdry_est} then yields
$$
g_{x}(\veps) = \frac1{\veps^d} \int^{1}_{0} \left(  \int_{B_{d-1}(0,\veps) \cap A^{\veps}_{t} } u(\hat y) - \veps t - f(\hat y) \; \Rd \hat y  \right)(1 +  t \veps H_{x} ) \; \Rd t + O(\veps^2).
$$
Now let $f^{\veps}(z) := \frac1{\veps} f(\veps z)$ and define the corresponding subset $C^{\veps}_{t}$ of $(0,1) \times B_{d-1}(0,1)$ as 
$$C^{\veps}_{t} := \left\{  (t,z) \in (0,1) \times B_{d-1}(0,1): -\sqrt{1 - \|z\|^2 } \leq  f^{\veps}(z) + t  \leq \sqrt{1 - \|z\|^2 }\right\},$$ 
then make the change of variables $\hat y = \veps z$ to see that
$$
g_{x}(\veps) = \int_{C^{\veps}_{t} } \left(\sqrt{1 - \|z\|^2} - t - f^{\veps}(z) \; \right)(1 +  t \veps H_{x} ) \; \Rd z \Rd t + O(\veps^2).
$$
Recalling \eqref{SecOrderf} shows that
\begin{equation}\label{eq:SecOrderf2}
f^{\veps}(z) = \frac{\veps}{2} \sum_{i=1}^{d-1} \kappa_i \langle z , v_i \rangle ^2  + O(\veps^2), 
\end{equation}
which then allows us to obtain an expansion of $g_{x}(\veps)$ in terms of $\veps$ according to the relation
\begin{align}\label{eq:g_expansion}
g_{x}(\veps) &= \int_{C^{\veps}_{t} } \left( \sqrt{1 - \|z\|^2} - t \right)\; \Rd t\Rd z \nonumber \\
& + \veps \int_{C^{\veps}_{t}} \left( tH_x(\sqrt{1 - \|z\|^2} - t) - \frac1{2} \sum_{i=1}^{d-1} \kappa_i \langle z , v_i \rangle ^2 \right) \; \Rd t \Rd z+ O(\veps^2).
\end{align}
The bias estimate \eqref{errorNonlocal} then directly follows after computing each of these terms individually.

We begin by considering the first term in the expansion, i.e.
$$
\mathrm{I} :=\int_{C^{\veps}_{t} } \left(  \sqrt{1 - \|z\|^2} - t  \right) \; \Rd t \Rd z.
$$
Given $\veps > 0$ and $z \in B_{d-1}(0,1)$ define $c(z) := \max\{ -\sqrt{1-\|z\|^2} -  f^{\veps}(z) , 0 \}$  and $C(z) := \min\{ \sqrt{1-\|z\|^2} -  f^{\veps}(z) , 1\},$ so that we may easily write
\begin{align*}
\mathrm{I} =  \int_{ B_{d-1}(0,1) } (C(z)-c(z))\left( \sqrt{1 - \|z\|^{2}} - \frac{C(z) + c(z)}{2}  \right) \;\Rd z.
\end{align*}
As the set where $c(z) \neq 0$ has measure at most $O(\veps^2),$ we easily conclude that
\begin{align*}
\mathrm{I} =  \int_{ B_{d-1}(0,1) } C(z)\left( \sqrt{1 - \|z\|^{2}} - \frac{C(z)}{2}  \right) \;\Rd z + O(\veps^2).
\end{align*}
If $C(z) = 1$ then $\sqrt{1 - \|z\|^{2}} - \frac{C(z)}{2} = \frac1{2}(1 - \|z\|^2) + O(\veps^2)$ as well. In any case, it follows that
\begin{align}\label{eq:first_comp}
\mathrm{I} = \frac1{2} \int_{B_{d-1}(0,1)} (1 - \|z\|^{2}) \; \Rd z + O(\veps^{2}) = \frac{\sigma_d}{2} +  O(\veps^{2}).
\end{align}
We now proceed to compute the second term in the expansion
\begin{align*}
\mathrm{II} := & H_{x} \int_{C^{\veps}_{t}} \left( t \sqrt{1 - \|z\|^2} - t^2 \right) \; \Rd t \Rd z \\
 = & H_{x} \int_{B_{d-1}(0,1)} C^2(z)\left( \frac{\sqrt{1-\|z\|}}{2} - \frac{C(z)}{3} \right) \; \Rd z + O(\veps^2)
\end{align*}
and the third term in the expansion
$$
\mathrm{III} := \frac1{2} \sum_{i=1}^{d-1} \kappa_i \int_{C^{\veps}_{t}} \langle z , v_i \rangle ^2 \; \Rd t \Rd z =  \frac1{2} \sum_{i=1}^{d-1} \kappa_i \int_{B_{d-1}(0,1)} \langle z , v_i \rangle ^2 C(z) \; \Rd z + O(\veps^2)
$$
in a similar fashion. We always have $C(z) = \sqrt{1-\|z\|^2} + O(\veps)$, so that
\begin{align}\label{eq:sec_comp}
\mathrm{II} = & \frac{H_{x}}{6} \int_{B_{d-1}(0,1)}  \!(1-\|z\|^2)^{3/2} \; \Rd z + O(\veps) \\ 
= & \frac{H_{x}\mathrm{vol}(\mathcal{S}^{d-2})}{6} \int^{1}_{0} (1- r^2)^{3/2}r^{d-2} \; \Rd r + O(\veps)
\end{align}
The third term follows similarly by appealing to spherical coordinates, in that we have
\begin{align*}
\mathrm{III} &=  \frac1{2} \sum^{d-1}_{i=1} \kappa_i \int_{B_{d-1}(0,1)} \sqrt{1-\|z\|^2}   \langle z , v_i \rangle ^2 \; \Rd z + O(\veps) \\
&=  \frac{H_x\mathrm{vol}(\mathcal{S}^{d-2})}{2(d-1)} \int^{1}_{0}  \sqrt{1- r^2} r^{d} \; \Rd r + O(\veps) = \mathrm{II} + O(\veps)
\end{align*}
thanks to an integration by parts in the final term. We therefore have that $\mathrm{I} = \sigma_{d}/2 + O(\veps^2)$ and $\mathrm{II} - \mathrm{III} = O(\veps),$ so that $g_{x}(\veps) = \sigma_{d}/2 + O(\veps^{2})$ and
\[ \NLTVenon{\Omega} = 2 \int_{\partial \Omega} \; g_{x}(\veps) \; \Rd \mathcal{H}^{d-1} = \sigma_{d} \Pem(\Omega) + O(\veps^{2})\]
as desired.

We may also show that when $\Omega$ is a fixed ball, say $\Omega = B_{d}(x_{c}, \frac13)$ for $x_c \in \mathbb{R}^{d}$ the center point of $\dom$, that the absolute value of the difference between $\NLTVenon{\Omega}$ and $\sigma_d \Pem(\Omega)$ remains bounded from below by $c\veps^2$ for $c > 0$ some positive constant. The proof proceeds similarly to the proof of the bias estimate above. In particular, this shows that the bound in Lemma \ref{pnlp} is optimal in terms of scaling for  general sets with smooth boundary.
\section{Proof of Lemma \ref{LemmaMomentsAveragephi}}
\label{AppendixMomentsPhi}

The proof follows the same argument used in the proof of Theorem 4.1 in \cite{GTS} or Theorem 6.2 in \cite{AB2}. We assume that $\dist(\Omega , \partial \D ) >0$. Such assumption implies that the perimeter of $\Omega$ with respect to $\D$, that is $\Pem(\Omega)$ defined in \eqref{perimeter}, is equal to the perimeter of $\Omega$ with respect to $\R^d$.  We remark that a slight modification of the argument we present below proves the result in the general case and hence we omit the details (see for example the proof of Theorem 4.1 in \cite{GTS}).

First we prove that for any function $u : \R^d \rightarrow [0,1]$ with $u \in W^{1,1}(\R^d) \cap C^\infty(\R^d)$ and for all $\veps>0$ we have 
\begin{equation}
\int_{\R^d}\left( \int_{\R^d} \frac{\1_{\|x-y\|\leq \veps}}{\veps^{d+1}}|u(y)- u(x)| \Rd y \right)^p \Rd x \leq \frac{\alpha_d^{p-1} \sigma_d}{\veps^{p-1}}  \int_{\R^d} \|\nabla u (x)\| \Rd x,
\label{aux10}
\end{equation}
Inequality \eqref{aux10} follows from
\begin{align*}
\begin{split}
\int_{\R^d} & \left( \int_{\R^d} \frac{\1_{\|x-y\|\leq \veps}}{\veps^{d+1}}|u(y)- u(x)| \Rd y \right)^p  \Rd x \\ 
& = \frac{1}{\veps^{p}} \int_{\R^d} \left(\int_{B_d(0,1)}| u(x+ \veps h) - u(x)|      \Rd h \right)^p \Rd x  \\
& \leq  \frac{\alpha_d^{p-1}}{\veps^p} \int_{\R^d}    \int_{B_d(0,1)} | u(x + \veps h ) - u(x)  |^p \Rd h \,   \Rd x  \\
& \leq  \frac{\alpha_d^{p-1}}{\veps^p} \int_{\R^d}    \int_{B_d(0,1)} | u(x + \veps h ) - u(x)  | \Rd h   \, \Rd x  \\
& =  \frac{\alpha_d^{p-1}}{\veps^{p-1}} \int_{\R^d}    \int_{B_d(0,1)} \left \lvert \int_{0}^{1} \nabla u(x+ t\veps h) \cdot  h \Rd t  \right \rvert \Rd h  \,  \Rd x  \\
& \leq  \frac{\alpha_d^{p-1}}{\veps^{p-1}} \int_{\R^d}    \int_{B_d(0,1)} \int_{0}^{1} \lvert \nabla u(x+ t\veps h)  \cdot h \rvert \Rd t  \, \Rd h    \, \Rd x \\
& =  \frac{\alpha_d^{p-1}}{\veps^{p-1}}  \int_{0}^{1}    \int_{B_d(0,1)} \int_{\R^d}  \lvert \nabla u (x)  \cdot h \rvert \Rd x    \, \Rd h \, \Rd t     \\
& =  \frac{\alpha_d^{p-1}}{\veps^{p-1}}  \int_{0}^{1}    \int_{\R^d}  \lVert \nabla u(x)\rVert \int_{B_d(0,1)}  \left \lvert \frac{\nabla u (x)}{\lVert \nabla u (x)\rVert}  \cdot h \right \rvert \Rd h  \,   \Rd x \,  \Rd t     \\
& =  \frac{\alpha_d^{p-1} \sigma_d}{\veps^{p-1}} \int_{\R^d} \| \nabla u (x) \| \Rd x
\end{split}
\end{align*}
where in the first equation we used the change of variables $h = \frac{x-y}{\veps}$, in the first inequality we used Jensen's inequality and in the second inequality the fact that $u$ takes values in $[0,1]$. 

Now, for any set $\Omega \subseteq \D$ as in the statement, we can find a sequence of functions $\left\{ u_k \right\}_{k \in \NN}$ with  $u_k: \R^d \rightarrow [0,1]$ , $u_k \in W^{1,1}(\R^d) \cap C^\infty(\R^d)$ and such that
\begin{equation}
 u_k \overset{L^1(\R^d)}{\longrightarrow} \1_{\Omega}, \quad
\lim_{k \rightarrow \infty} \int_{\R^d} \| \nabla u_k(x) \| \Rd x = \Pem(\Omega). 
\label{aux1molli}
\end{equation}
Such sequence can be obtained for example with the aid of standard mollifiers (see Theorem 13.9 in \cite{Leoni}). It follows from \eqref{aux10} and from \eqref{aux1molli} that 
\begin{equation*}
\int_{\R^d}\left( \int_{\R^d} \frac{\1_{\|x-y\|\leq \veps}}{\veps^{d+1}}|\1_\Omega(y)- \1_\Omega(x)| \Rd y \right)^p \Rd x \leq \frac{\alpha_d^{p-1} \sigma_d}{\veps^{p-1}}  \Pem(\Omega).
\end{equation*}
Finally, notice that
$$\int_{\D} \bar{\phi}^p_{\veps_n}(x) \Rd x  \leq   \int_{\R^d}\left( \int_{\R^d} \frac{\1_{\|x-y\|\leq \veps}}{\veps^{d+1}}|\1_\Omega(y)- \1_\Omega(x)| \Rd y \right)^p \Rd x \leq \frac{\alpha_d^{p-1} \sigma_d}{\veps^{p-1}}  \Pem(\Omega).
  $$

\section{Proof of Lemma \ref{LemmaVarianceg}}
\label{AppendixVarianceg}
The proof is based on similar computations to the ones in Appendix \ref{AppendixLemmaBias} and thus we simply highlight the main ideas. First of all note that
\begin{equation*}
\Var  \left(g_{n,1}(X_1) \right)= \int_\D \bar{\phi}_{\veps_n}^2(x) dx - \left( \Pem_{\veps_n}(\Omega) \right)^2.
\end{equation*}
Since $\Pem_{\veps_n}(\Omega)=O(1)$, our task reduces to computing the integral in the above expression. It follows from \eqref{phiAverageDef} that for all $\veps_n $ small enough (so that $T_{\veps_n} \subseteq \D$),
$$  \int_\D \bar{\phi}_{\veps_n}^2(x) \Rd x  =   \frac{1}{\veps^{2(d+1)}_n}   \int_{\tubni}|B_{d}(x,\veps_n) \cap \Omega^c |^2 \Rd x +\frac{1}{\veps^{2(d+1)}_n}  \int_{\tubno}|B_{d}(x,\veps_n) \cap \Omega |^2 \Rd x.   $$
We compute the first of the integrals from the above expression.  As in the proof of Lemma \ref{pnlp}, we write
\begin{align*}
\frac1{\veps^{2(d+1)}_n} &  \int_{\tubni}| B_{d}(x,\veps_n) \cap \Omega^c|^2 \; \Rd x \\
& = \frac{1}{\veps_n}  \int_{\partial \Omega}\left( \frac{1}{\veps^{2d}_n}\int_{0}^{1} |B_{d}(x-t\veps_n \textbf{N}(x),\veps_n) \cap \Omega^c|^2 \det(I + t\veps_n\textbf{S}_x) \; \Rd t\right) \; \Rd \mathcal{H}^{d-1}(x).
\end{align*}
For $x \in \partial \Omega$ we study the expression
$$h_{\veps_n}(x) := \int_{0}^{1} \frac{| B_d(x- t \veps_n \textbf{N}(x)  , \veps_n) \cap \Omega^c|^2  }{\veps_n^{2d}} \det(I + t \veps_n  \textbf{S}_x) \Rd t.$$
Note that $\det(I + t \veps_n \textbf{S}_x) = (1 + t\veps_n \kappa_1) \cdot\cdot\cdot (1+ t \veps_n \kappa_{d-1}) = 1 + O(\veps_n) $, where $\kappa_1,\dots, \kappa_{d-1}$ are the principal curvatures (eigenvalues of the shape operator $\textbf{S}_x$). Hence, 
$$  h_{\veps_n}(x) = \int_{0}^{1} \frac{| B_d(x - t \veps_n \textbf{N}(x), \veps_n) \cap \Omega^c  |^2}{\veps_n^{2d}} \Rd t  + O(\veps_n) . $$ 
Without loss of generality, we may assume that $x= 0$, that $\textbf{N}(x)=e_d$ and that around  $x$ the boundary $\partial \Omega$ coincides with the graph $\hat x = (x_1, \dots, x_{d-1}) \mapsto (\hat{x}, f(\hat{x})) \in \mathbb{R}^{d}$  of a smooth function $f(\hat x)$ that satisfies both $f(0)=0$ and $\nabla f (0) =0$ simultaneously, and we denote by $v_1, \dots , v_{d-1}$ the eigenvectors of $\textbf{S}_x$ (just as in Appendix \ref{AppendixLemmaBias}). Then,
 $f : B_{d-1}(0,\veps_n) \rightarrow \R$, satisfies  $f(\hat{x}) = \sum_{i=1}^{d-1}\kappa_i \left<  \hat{x}, v_i\right> ^2 + O(\veps_n^3)  $.   

Now for fixed $t \in [0,1]$ we define $H_{1-t}$ to be the hyperplane
$$ H_{1-t} := \left\{ x= (\hat{x}, x_d) \: : \: x_d \geq 1-t \right\}, $$  
and we let
$$ A_t:= |B_d(0,1) \cap H_{1-t} |. $$
With these definitions we can write
\begin{align}
\begin{split}
 | B_d(x - t \veps_n \textbf{N}(x), \veps_n) \cap \Omega^c  |&= \veps_n^d A_{t}  +\sum_{i=1}^{d-1} \kappa_i \int_{B_{d-1}(0,\veps_n \sqrt{1-t})}   \left<  \hat{x}, v_i\right> ^2 \Rd \hat{x} + O(\veps_n^{d+2}) \\
 &= \veps_n^d A_t  +  \frac{\sum_{i=1}^{d-1} \kappa_i}{d-1}  \int_{B_{d-1}(0,\veps_n \sqrt{1-t})}   \|\hat{x}\|^2 \Rd \hat{x} + O(\veps_n^{d+2})\\
 &= \veps_n^d A_{t} + O(\veps_n^{d+1}),
\end{split}
\end{align}
where the second equality holds due to symmetry, and where the last equality follows after computing $\int_{B_{d-1}(0,\veps_n \sqrt{1-t})}   \|\hat{x}\|^2 \Rd \hat{x}$  using polar coordinates. From the above, it follows that
$$ \frac{ | B_d(x - t \veps_n \textbf{N}(x), \veps_n) \cap \Omega^c  |^2}{\veps_n^{2d}} = A_{t}^2 + O(\veps_n). $$
We conclude that 
$$h_{\veps_n}(x) =   \int_{0}^{1}A_{t}^2 dt + O(\veps_n). $$
Therefore, 
$$ \frac1{\veps_n^{2(d+1)}}  \int_{T_{\veps_{n}}^-}| B_{d}(x,\veps_n) \cap \Omega^c|^2 \; \Rd x   = \frac{1}{\veps_n} \int_{0}^{1}A_{t}^2 \Rd t  \Pem(\Omega) + O(1).  $$
Analogously, we can obtain a similar expression for $\frac1{\veps_n^{2(d+1)}}  \int_{T_{\veps_{n}}^+}| B_{d}(x,\veps_n) \cap \Omega|^2 \; \Rd x$ and from this we deduce \eqref{Varianceg1}.


\subsection*{Acknowledgements}
DS and NGT are grateful to  NSF for its support (grants DMS-1211760 and DMS-1516677). JvB was supported by NSF grant DMS 1312344/DMS 1521138. The authors are grateful to ICERM (supported by NSF grant 0931908), where part of the research was done during the research cluster: \emph{Geometric analysis methods for graph algorithms}.
Furthermore they are grateful to Ki-Net (NSF Research Network Grant RNMS11-07444) for opportunities provided.
 The authors would like to thank the Center for Nonlinear Analysis of the Carnegie Mellon University for its support.

\bibliography{BiblioPointwise}
\bibliographystyle{siam}


\end{document}